\documentclass[a4paper,final]{article}

\usepackage{graphicx}

\usepackage{amssymb}

\usepackage{amsthm}
\newtheorem{theorem}{Theorem}
\newtheorem{proposition}[theorem]{Proposition}
\newtheorem{lemma}[theorem]{Lemma}
\newtheorem{corollary}[theorem]{Corollary}

\theoremstyle{definition}
\newtheorem{remark}[theorem]{Remark}

\usepackage{url}

\newcommand{\calT}{{\cal T}}
\newcommand{\interior}[1]{{\rm Int}(#1)}

\newcommand{\ptwoirred}{$\mathbb{P}^2$-irreducible}
\newcommand{\matRP}{\mathbb{RP}}

\newcommand{\Tone}{\mbox{${\rm T}_1$}}
\newcommand{\Ttwo}{\mbox{${\rm T}_2$}}
\newcommand{\Ti}{\mbox{${\rm T}_*$}}

\begin{document}

\title{Moves for standard skeleta\\ of 3-manifolds with marked boundary}

\author{Gennaro Amendola}

\maketitle

\begin{abstract}
A 3-manifold with marked boundary is a pair $(M,X)$, where $M$ is a compact 3-manifold whose (possibly empty) boundary is made up of tori and Klein bottles, and $X$ is a trivalent graph that is a spine of $\partial M$.
A standard skeleton of a 3-manifold with marked boundary $(M,X)$ is a standard sub-polyhedron $P$ of $M$ such that $P \cap \partial M$ coincides with $X$ and with $\partial P$, and such that $P \cup \partial M$ is a spine of $M \setminus B$ (where $B$ is a ball).

In this paper, we will prove that the classical set of moves for standard spines of 3-manifolds ({\em i.e.}~the MP-move and the V-move) does not suffice to relate to each other any two standard skeleta of a 3-manifold with marked boundary.
We will also describe a condition on the 3-manifold with marked boundary that tells whether the generalised set of moves, made up of the MP-move and the L-move, suffices to relate to each other any two standard skeleta of the 3-manifold with marked boundary.

For the 3-manifolds with marked boundary that do not fulfil this condition, we give three other moves: the CR-move, the \Tone-move and the \Ttwo-move.
The first one is local and, with the MP-move and the L-move, suffices to relate to each other any two standard skeleta of a 3-manifold with marked boundary fulfilling another condition.
For the universal case, we will prove that the non-local \Tone-move and \Ttwo-move, with the MP-move and the L-move, suffice to relate to each other any two standard skeleta of a generic 3-manifold with marked boundary.

As a corollary, we will get that disc-replacements suffice to relate to each other any two standard skeleta of a 3-manifold with marked boundary.
\end{abstract}

\begin{center}
\end{center}

\begin{center}
\end{center}

\section*{Introduction}

In~\cite{Matveev:compl_def} Matveev defined for any
compact 3-manifold $M$ a non-negative integer $c(M)$, which he called
the {\em complexity} of $M$.
The complexity function $c$ has the following remarkable properties:
it is additive under connected sum, it does not increase when cutting
along incompressible surfaces, and it is finite-to-one on the most
interesting classes of 3-manifolds.
Namely, among the compact 3-manifolds having complexity $c$ there is
only a finite number of closed \ptwoirred\ ones and a
finite number of finite-volume hyperbolic ones (with cusps and/or with
compact geodesic boundary).
The complexity of a closed \ptwoirred\ 3-manifold is then precisely the minimal number of
tetrahedra needed to triangulate it, except when its complexity is $0$, {\em
i.e.}~when it is $S^3$, $\matRP^3$ or $L_{3,1}$.

The problem of computing (or at least estimating) the complexity of $M$ naturally arose.
For the closed case, Martelli and Petronio developed a theory of decomposition of closed \ptwoirred\ 3-manifolds~\cite{Martelli-Petronio:decomposition}.
The decomposition is made along tori and Klein bottles (as in the JSJ
decomposition) in such a way that the complexity of the original
manifold is the sum of the (suitably defined) complexities of the
building blocks (called bricks).
The bricks carry an extra structure given by a finite set of trivalent
graphs, each contained in a boundary component so that the complement is a disc.
These graphs are fundamental because they affect both the definition
of the complexity of bricks and the reassembling of bricks.

This theory seems to be very useful for the computation/estimation process.
Martelli and Petronio~\cite{Martelli-Petronio:complexity_9} used the orientable version of it to list closed
irreducible orientable 3-manifolds of complexity up to 9, and then Martelli~\cite{Martelli:complexity_10} used it to get the list up to complexity 10.
Moreover, in~\cite{Martelli-Petronio:estimation_complexity}, Martelli and Petronio gave estimations on the complexity of closed 3-manifolds.
There are many reasons making this theory feasible, at least in
the orientable case up to complexity 10: for instance, there are very
few bricks with respect to closed manifolds, they must satisfy
many topological restrictions (so the search of bricks is easier than
that of closed manifolds), they can be assembled (to produce closed
manifolds) in a finite (small) number of ways, and the decomposition
into bricks seems to be a refinement of the JSJ decomposition (so it is easy to
give a ``name'' to the manifolds obtained by assembling bricks).

The main objects of the decomposition theory are the 3-manifolds with marked boundary.
A 3-manifold with marked boundary is a pair $(M,X)$, where $M$ is a compact 3-manifold whose (possibly empty) boundary is made up of tori and Klein bottles, and $X$ is a trivalent graph that is a spine of $\partial M$.
The main tools used in the decomposition to work on 3-manifolds with marked boundary are the standard skeleta.
A standard skeleton of a 3-manifold with marked boundary $(M,X)$ is a standard sub-polyhedron $P$ of $M$ such that $P \cap \partial M$ coincides with $X$ and with $\partial P$, and such that $P \cup \partial M$ is a spine of $M \setminus B$ (where $B$ is a ball).
We remark that standard skeleta are viewed up to isotopy.

These objects may seem less natural than those used by Turaev and Viro in~\cite{Turaev-Viro}, {\em i.e.}~the standard spines $P$ of $M$ such that $P \cap \partial M$ coincides with $X$ and with $\partial P$.
However, standard skeleta are very useful when one wants to glue two 3-manifolds with marked boundary along the boundary by identifying the trivalent graphs, because after the gluing the skeleta unite and form a standard skeleton of the manifold with marked boundary obtained.
On the contrary, the spines used by Turaev and Viro would unite and form a spine of the manifold minus a ball.

In this paper, we will deal with the problem of deciding how different standard skeleta of a 3-manifold with marked boundary are related to each other.
We will prove that the classical set of moves for standard spines of 3-manifolds ({\em i.e.}~the MP-move and the V-move) does not suffice to relate to each other any two standard skeleta of a 3-manifold with marked boundary.
The reason is that these moves are very ``local'', {\em i.e.}~the portion of the skeleton involved in the move is contained in a ``small'' ball.
A first solution to this problem is to use the L-move, a generalisation of the V-move.
However, this suffices only for a particular class of 3-manifolds with marked boundary.
We will define an object for each standard skeleton of a 3-manifold with marked boundary that is invariant under MP-moves and L-moves.
This object is the main ingredient for giving a condition on the 3-manifold with marked boundary that tells whether the MP-move and the L-move suffice to relate to each other any two standard skeleta of the 3-manifold with marked boundary.

If $M$ has $n$ boundary components, an {\em octopus} $o$ in $M$ is the image (viewed up to isotopy) of an embedding of the cone on $n$ points such that the preimage of each boundary component of $M$ is exactly one endpoint of the cone.
Note that any two omotopy equivalent octopuses can be obtained from each other by means of changes of crossings.
Instead, a modification of the homotopy type of an octopus needs a more drastic move.

If $P$ is a standard skeleton of $(M,X)$, we consider the ideal triangulation dual to $P \cup \partial M$; the edges dual to the regions in the boundary components of $M$ form an octopus defined unambiguously from $P$.
The main point is that the octopus of $P$ is invariant under MP-moves and L-moves.

We will prove that if in a 3-manifold with marked boundary there is only one octopus, the MP-move and the L-move suffice to relate to each other any two standard skeleta of the 3-manifold with marked boundary.
For the 3-manifolds with marked boundary that do not fulfil this condition, we give three other moves that can change the octopus: the CR-move, the \Tone-move and the \Ttwo-move.
They are particular types of the so-called disc replacement moves, which have been defined by Matveev (see~\cite{Matveev:new:book}).
The first one is ``local'' and allows us to change a crossing of the octopus.
We will prove that if any two octopuses of a 3-manifold with marked boundary can be obtained from each other by changes of crossing, then any two standard skeleta of the manifold can be related to each other by MP-moves, L-moves and CR-moves.
For the universal case, we will use the non-local \Ti-moves, which allow us to change the octopuses arbitrarily, and we will prove that the MP-move, the L-move and the \Ti-moves suffice to relate to each other any two standard skeleta of a generic 3-manifold with marked boundary.

As a corollary, we will get that disc replacement moves (indeed, only some of them that we will call disc-replacements) suffice to relate to each other any two standard skeleta of a 3-manifold with marked boundary.

\section{Preliminaries}

Throughout this paper, $M$ will be a fixed compact connected 3-manifold with (possibly empty) boundary made up of tori and Klein bottles.
We will adopt an ``embedded viewpoint'', {\em i.e.}~$M$ is supposed to be fixed and every object in $M$ is viewed up to isotopy in $M$.
Using the {\em Hauptvermutung}, we will freely intermingle the differentiable,
piecewise linear and topological viewpoints.

\subsection{Manifolds with marked boundary}

\paragraph{Spines of surfaces}
If $C$ is a connected surface, we call a {\em spine} of $C$ a trivalent graph $X$ contained in $C$ such that $C \setminus X$ is an open disc.
If $C$ has $n$ connected components, a spine of $C$ is a collection of $n$ spines, one for each component of $C$.
Spines of surfaces are viewed up to isotopy.
With an easy Euler-characteristic argument, it can be proved that if $C$ is a torus $T$ or a Klein bottle $K$, a spine of $C$ must be a connected trivalent graph with two vertices.
Note that there are precisely two such graphs: in Fig.~\ref{2vert_spines_surfaces:fig} we have shown the two graphs, say $\theta$ and $\sigma$, and their embeddings in $C$.
\begin{figure}
  \centerline{\includegraphics{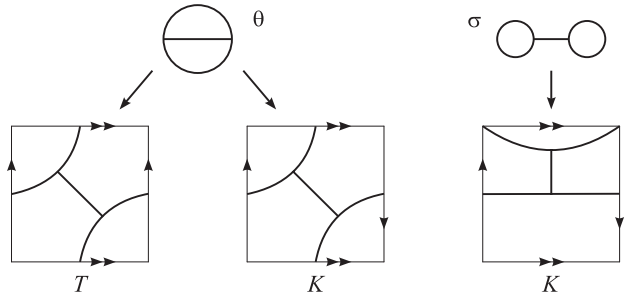}}
  \caption{The spines of the torus $T$ and the Klein bottle $K$.}
  \label{2vert_spines_surfaces:fig}
\end{figure}
Note that $\theta$ is a spine of both the torus and the Klein bottle, while $\sigma$ is a spine of the Klein bottle only.
Note also that the image of the embedding of $\theta$ in the torus is not unique (also up to isotopy), while the images of the embeddings of both $\theta$ and $\sigma$ in the Klein bottle are unique (up to isotopy)~\cite{Martelli-Petronio:decomposition}.

\paragraph{Manifolds with marked boundary}
A pair $(M,X)$ is said to be a {\em manifold with marked boundary} if $X$ is a spine of $\partial M$.
Hence, we have $\partial M = \sqcup_{i=1}^n C_i$, where each $C_i$ is a torus or a Klein bottle, and $X = \{X_1,\ldots,X_n\}$, where $X_i \subset C_i$ so that $C_i \setminus X_i$ is a disc.
As said above, $M$ is considered fixed, while $X$ is viewed up to isotopy.

\subsection{Spines, skeleta and ideal triangulations}

\paragraph{Standard polyhedra}
A {\em quasi-standard} polyhedron $P$ is a finite, connected and purely $2$-dimensional polyhedron in which each point has a neighbourhood
of one of the types shown in Fig.~\ref{quasi_standard:fig}.
\begin{figure}
  \centerline{\includegraphics{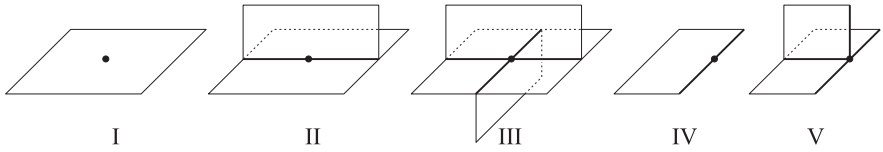}}
  \caption{The five typical neighbourhoods in a quasi-standard polyhedron.}
  \label{quasi_standard:fig}
\end{figure}
The {\em boundary} $\partial P$ of $P$ is the trivalent graph made up of the points of type IV and V.
The set of points of type II and III (the {\em singular points}) is denoted by $S(P)$.
A quasi-standard polyhedron is called {\em standard} if it is cellularized by singularity and boundary.
In a standard polyhedron, the points of type III are called {\em vertices}, the connected components of the set of the points of type II are called
{\em edges}, and the connected components of the set of the points of type I (the {\em non-singular points}) are called {\em regions}.
In the figures, the singular set and the boundary of the polyhedron are drawn thick, and the vertices are marked by a thick dot.

\paragraph{Spines and skeleta}
A sub-polyhedron $P$ of a manifold $M$ with non-empty boundary is called a {\em spine} of $M$ if $M$ collapses to it.
If $M$ is closed, the boundary can be created by puncturing $M$ ({\em i.e.}~by considering $M$ minus a ball).
It is by now well-known, after the work of Casler~\cite{Casler}, that a standard spine without boundary determines unambiguously $M$ up to homeomorphism and that
every $M$ has standard spines.
In the literature it is a customary convention that the spine should embed in $\interior{M}$, but this is not essential, so we allow spines to
embed in the whole of $M$.

A {\em skeleton} $P$ of a manifold with marked boundary $(M,X)$ is a quasi-standard sub-polyhedron $P$ of $M$ such that $P \cap \partial M$
coincides with $X$ and with $\partial P$, and such that $P \cup \partial M$ is a spine of $M \setminus B$ (where $B$ is a ball).
Each skeleton of $(M,X)$ is always viewed up to isotopy.
Note that $P \cup \partial M$ has no boundary, and that if $M$ is closed ({\em i.e.}~$X = \emptyset$), a standard skeleton of $(M,\emptyset)$ is just a standard spine without boundary of
$M$.
We will prove below that every manifold with marked boundary has standard skeleta (Lemma~\ref{existence:lem}).

\begin{remark}
For the sake of clarity, we mention that our notion of skeleton is different from the Turaev-Viro one~\cite{Turaev-Viro}, which has been described in Introduction above.
Actually, our notion of skeleton is less general than the Martelli-Petronio one~\cite{Martelli-Petronio:decomposition} because they allow points to have as regular neighbourhood the cone on any compact subset of the circle with three radii; however, our notion of standard skeleton is equal to that of~\cite{Martelli-Petronio:decomposition}.
\end{remark}

\paragraph{Ideal triangulations and duality}
An {\em ideal triangulation} of a manifold $M$ with non-empty boundary is a partition $\calT$ of $\interior{M}$ into open cells of dimensions 1, 2 and 3, induced by a triangulation $\widehat{\calT}$ of the space $\widehat{M}$, where:

\begin{itemize}

\item $\widehat{M}$ is obtained from $M$ by collapsing to a point each component of $\partial M$;

\item $\widehat{\calT}$ is a triangulation only in a loose sense, namely self-adjacencies and multiple adjacencies of tetrahedra are
allowed;

\item the vertices of $\widehat{\calT}$ are precisely the points of $\widehat{M}$ corresponding to the components of $\partial M$, so $\widehat{M}$ minus those vertices can be identified with $\interior{M}$.

\end{itemize} 

It turns out~\cite{Matveev-Fomenko,Matveev:new:book} that there exists a natural bijection between standard spines
without boundary and ideal triangulations of a manifold.
Given an ideal triangulation $\calT$, the corresponding standard spine without boundary $P$ is just the 2-skeleton of the dual cellularization, as
illustrated in Fig.~\ref{duality:fig}.
\begin{figure}
  \centerline{\includegraphics{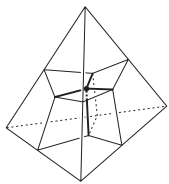}}
  \caption{Portion of spine dual to a tetrahedron of an ideal triangulation.}
  \label{duality:fig}
\end{figure}
The inverse of this correspondence is denoted by $P \mapsto \calT(P)$, and $\calT(P)$ is said to be the ideal triangulation {\em dual to
$P$}.

\subsection{First moves}\label{first_moves:subsec}

In this section, we will describe the moves giving the calculus for standard spines without boundary.
These moves will be fundamental for the generalisation of this calculus to standard skeleta.

\paragraph{MP-move}
Let us start from the move shown in Fig.~\ref{MP_V_moves:fig}-left, which is called an MP-{\em move}.
\begin{figure}
  \centerline{\includegraphics{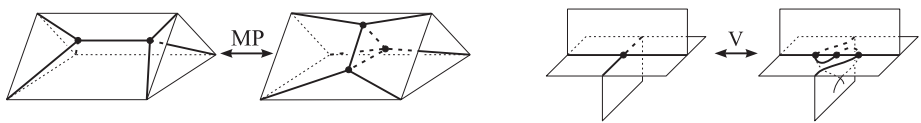}}
  \caption{The MP-move (on the left) and the V-move (on the right).}
  \label{MP_V_moves:fig}
\end{figure}
Such a move will be called {\em positive} if it increases (by one) the number of vertices, and {\em negative} otherwise.
Note that if we apply an MP-move to a standard skeleton of $(M,X)$, the result will be another standard skeleton of $(M,X)$.
It is already known (see Theorem~\ref{MP_calculus:teo} below) that any two standard spines without boundary of the same $M$ with at least two
vertices can be transformed into each other by MP-moves.

\paragraph{V-move}
If one of the two standard spines without boundary of $M$ (we want to transform into each other) has just one vertex, another move is required.
The move shown in Fig.~\ref{MP_V_moves:fig}-right is called a V-{\em move}.
Note that if we apply such a move to a standard skeleton of $(M,X)$, the result will be another standard skeleton of $(M,X)$.
As above, we have {\em positive} and {\em negative} V-moves.

If a positive V-move is applied to a standard spine without boundary with at least two vertices, the V-move is a composition of
MP-moves.
In Fig.~\ref{V_comp_MP:fig} we show the three positive and the negative MP-moves giving the V-move.
\begin{figure}
  \centerline{\includegraphics{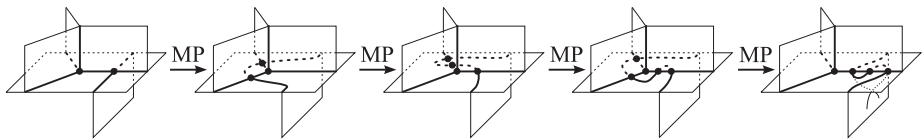}}
  \caption{If there is another vertex, each positive V-move is a composition of MP-moves.}
  \label{V_comp_MP:fig}
\end{figure}

On the contrary, if a V-move is applied to a standard skeleton (also with many vertices), the fact that the V-move is a composition of MP-moves
may not be true: in fact, it may occur that no edge starting from the vertex on which we apply the V-move ends in another vertex ({\em i.e.}~all the edges starting from the vertex end in the boundary or in the vertex itself).
For instance, consider the manifold with marked boundary $B_3$ of~\cite{Martelli-Petronio:complexity_9}, {\em i.e.}~$\left(T\times I, \left\{X_0,X_1\right\}\right)$ where $X_0$ and $X_1$ are related by a flip, and its standard skeleton with one vertex (see Fig.~\ref{manifold_one_vert:fig}).
\begin{figure}
  \centerline{\includegraphics{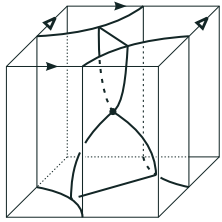}}
  \caption{The 3-manifold with marked boundary $B_3$ and its standard skeleton with one vertex (the lateral faces should be identified).}
  \label{manifold_one_vert:fig}
\end{figure}

\paragraph{The calculus for standard spines without boundary}
It is already known, after the work of Matveev~\cite{Matveev:calculus} and Piergallini~\cite{Piergallini}, that the moves described above give a
calculus for standard spines without boundary.
Namely, we have the following.

\begin{theorem}[Matveev-Piergallini]\label{MP_calculus:teo}
Any two standard spines without boundary of $M$ can be obtained from each other via a sequence of {\rm V}-~and {\rm MP}-moves.
If moreover the two spines have at least two vertices, then they can be obtained from each other via a sequence of {\rm MP}-moves
only.
\end{theorem}

\paragraph{L-move}
A generalisation of the V-move is the L-{\em move}, see Fig.~\ref{L_move:fig}.
\begin{figure}
  \centerline{\includegraphics{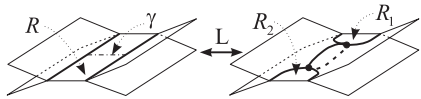}}
  \caption{The L-move.}
  \label{L_move:fig}
\end{figure}
As above, we have {\em positive} and {\em negative} L-moves.
As opposed to the V-move, this move is non-local, so it must be described with some care.
A positive L-move, which increases by two the number of vertices, is determined by an arc $\gamma$ disjoint from $\partial P$ and properly
embedded in a region $R$ of $P$.
The move acts on $P$ as in Fig.~\ref{L_move:fig}, but, in order to define its effect unambiguously, we must specify which pairs of regions, out of the
four regions incident to $R$ at the endpoints of $\gamma$, will become adjacent to each other after the move.
This is achieved by noting that we can choose a transverse orientation for the regular neighbourhood of $\gamma$ in $R$.
Using it, at each endpoint of $\gamma$ we can tell from each other the two regions incident to $R$ as being an upper and a lower one, and we
can stipulate that the two upper regions will become adjacent after the move (and similarly for the lower ones). 

For the negative case the situation is more complicated.
A negative L-move can lead to a non-standard polyhedron.
If $R_1$ and $R_2$ belong to the same region, after the negative L-move, the ``region'' $R$ would not be a disc.
To avoid this loss of standardness, we will call negative L-moves only those preserving standardness.
So a negative L-move can be applied only if the regions $R_1$ and $R_2$ are different.
With this convention, if we apply an L-move to a standard skeleton of $(M,X)$, the result will be another standard skeleton of
$(M,X)$.

\begin{remark}
A region $R$ of $P$ that is a rectangle incident to two edges of $\partial P$ cannot be modified via V-~and MP-moves, so we will need L-moves.
Consider, for instance, $T\times I$ (or $K\times I$) and its skeleton $X\times I$, where $X$ is a spine of $T$ (or $K$).
\end{remark}

\section{A partial calculus}\label{sec:partial_calculus}

Since a standard skeleton of a manifold with marked boundary $(M,\emptyset)$ is actually a standard spine without boundary of the manifold
$M$, a calculus for standard skeleta of $(M,\emptyset)$ is already known (see Theorem~\ref{MP_calculus:teo}).
For this reason, we will consider only manifolds with non-empty boundary, and, from now on, $(M,X)$ will be a manifold with marked boundary with $X = \{X_1,\ldots,X_n\} \neq \emptyset$.

\subsection{Octopus}

We have already noted that both L-~and MP-moves preserve the property of being a standard skeleton of $(M,X)$.
But there is also an invariant of standard skeleta of $(M,X)$ unchanged by L-~and MP-moves.
It is just this invariant the reason why L-~and MP-moves are not enough to obtain all the standard skeleta of $(M,X)$ from a fixed one.
In this section we will define this invariant.

\paragraph{Octopus}
An {\em octopus} $o$ in $M$ is the image of an embedding of the cone over $n$ points in $M\setminus\cup_{i=1}^{n}X_i$, such that the preimage of each boundary component of $M$ is exactly one endpoint of the cone.
More precisely, $o$ is the union $\cup_{i=1}^n \tau_i([0,1])$ of simple arcs $\tau_i : [0,1] \rightarrow M$ such that the $\tau_i$'s are disjoint
except for an endpoint that is in common ({\em i.e.}~$\tau_i([0,1]) \cap \tau_j([0,1]) = \{\tau_i(0)\}$, for $i \neq j$) and such that each $\tau_i$ has the other
endpoint on the component $C_i$ of $\partial M$ minus $X_i$ ({\em i.e.}~$\tau_i([0,1]) \cap \partial M = \{\tau_i(1)\} \subset C_i \setminus X_i$, for
$i=1,\ldots,n$).
Each $\tau_i$ will be called a {\em tentacle of $o$}, and the common endpoint $\tau_*(0)$ of the tentacles will be called the {\em head of $o$}.
As for skeleta, octopuses are viewed up to isotopy.

Now, we are able to define the invariant.
Let $P$ be a standard skeleton of $(M,X)$.
Recall that, by definition, $P \cup \partial M$ is a standard spine without boundary of $M \setminus B$, where $B$ is a ball.
Consider the ideal triangulation $\calT(P \cup \partial M)$ of $M \setminus B$.
The polyhedron $P \cup \partial M$ is obtained from $P$ by adding $n$ regions $C_i \setminus X_i$ (with $i=1,\ldots,n$).
Let us call $\alpha_i$ the edge of $\calT(P \cup \partial M)$ dual to the region $C_i \setminus X_i$.
The union of the arcs $\alpha_i$ is an octopus, defined unambiguously from $P$ (up to isotopy).
It will be called the {\em octopus of $P$} and will be denoted by $o(P)$.
Note that $P$ is a standard spine (with boundary) of $M \setminus N(o(P))$, where $N(o(P))$ is a regular neighbourhood of the octopus $o(P)$.

\begin{remark}\label{one_octopus:rem}
If $M$ has only one boundary component ({\em i.e.}~$n=1$), then $(M,X)$ has only one octopus.
\end{remark}

\begin{remark}\label{invariance_octopus:rem}
Let $P$ be a standard skeleton of $(M,X)$ with octopus $o(P)$.
We have already noted (see Section~\ref{first_moves:subsec}) that if we apply an L- or an MP-move to $P$, we obtain another standard
skeleton $P'$ of $(M,X)$.
Both $P$ and $P'$ are spines of $M \setminus N(o(P))$, so the octopuses $o(P)$ and $o(P')$ are equal.
Therefore, if $(M,X)$ has more than one octopus, the L- and the MP-move do not suffice to give a set of moves for standard skeleta of $(M,X)$.
\end{remark}

\paragraph{Existence of a standard skeleton for each octopus}
For each octopus $o$ in $(M,X)$ there exists a standard skeleton $P$ of $(M,X)$ such that $o=o(P)$.
This fact can be obviously deduced from the following lemma.
Note that the extra arc in the statement below is not necessary for proving the existence of standard skeleta, but we state the lemma in this form because it will be useful afterwards.
\begin{lemma}\label{existence:lem}
Let $o$ be an octopus in $(M,X)$, and let $\varphi : [0,1] \rightarrow \interior{M}$ be a simple loop starting from the head of $o$ such that $\varphi([0,1]) \cap o =
\{\varphi(0)=\varphi(1)\}$.
Then, there exists a standard skeleton $P$ of $(M,X)$ such that $o=o(P)$ and $\varphi([0,1])$ is an edge of the ideal triangulation $\calT(P \cup
\partial M)$ of $M$ minus a ball.
\end{lemma}

\begin{proof}
Let us consider regular neighbourhoods $N(o)$ of the octopus $o$ and $N(\varphi)$ of $\varphi([0,1])$, such that $N(o) \cup N(\varphi)$ is a
regular neighbourhood of $o \cup \varphi([0,1])$.
Let $Q$ be a standard spine without boundary of $M \setminus (N(o) \cup N(\varphi))$ contained in $\interior{M \setminus (N(o) \cup
N(\varphi))}$.
Note that we have a retraction $\pi$ of $M \setminus (N(o) \cup N(\varphi))$ onto $Q$.
Let $D$ be a disc properly embedded in $N(\varphi)$ intersecting $\varphi([0,1])$ transversely in one point.
Now, we can suppose that, by projecting $\partial D$ and $X$ to $Q$ along $\pi$, we obtain $\partial D \times [0,1)$ and $X \times
[0,1)$.
Up to isotopy, we can also suppose that both $\pi(\partial D)$ and all $\pi(X_i)$'s intersect $S(Q)$, and that $\pi(\partial D \cup
X)$ is transverse to $S(Q)$ and to itself.
Let us define $P$ as the union of $Q$, the disc $D$, the annulus $\partial D \times [0,1)$ and the $X_i \times [0,1)$'s.
The polyhedron $P$ is the skeleton we are looking for: in fact, $P$ is standard, $P \cap \partial M$ coincides with $X$ and with $\partial P$, $P \cup \partial M$ is a standard
spine of $M$ minus a ball, and $\varphi([0,1])$ coincides with the edge dual to the region $D \cup (\partial D \times [0,1))$ of
$P$.
\end{proof}

\subsection{Super-standard skeleta}

In this section we will describe a technical result that will be useful afterward.
A standard skeleton $P$ of $(M,X)$ will be called {\em super-standard} if $P = Q \cup (X \times [0,1))$, where $Q$ is a standard polyhedron
without boundary and $X \times [0,1)$ is made up of the regions of $P$ incident to $\partial P$.
For the sake of clarity, we note that our definition of super-standard skeleton is slightly different from the one of~\cite{Martelli-Petronio:decomposition}.

\begin{lemma}\label{super_standard:lem}
Each standard skeleton $P$ of $(M,X)$ can be transformed into a super-standard one via {\rm L}-~and {\rm MP}-moves.
\end{lemma}

Before proving the lemma, we describe another move on standard skeleta useful in the proof.
We call a C{\em -move} the move shown in Fig.~\ref{C_move:fig}.
\begin{figure}
  \centerline{\includegraphics{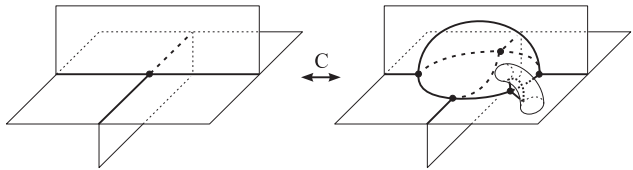}}
  \caption{The C-move.}
  \label{C_move:fig}
\end{figure}
As for the other moves, we have {\em positive} and {\em negative} C-moves.
Each positive C-move is a composition of V-~and MP-moves: the V-move and the (four) MP-moves are shown in Fig.~\ref{C_comp_V_MP:fig}.
\begin{figure}
  \centerline{\includegraphics{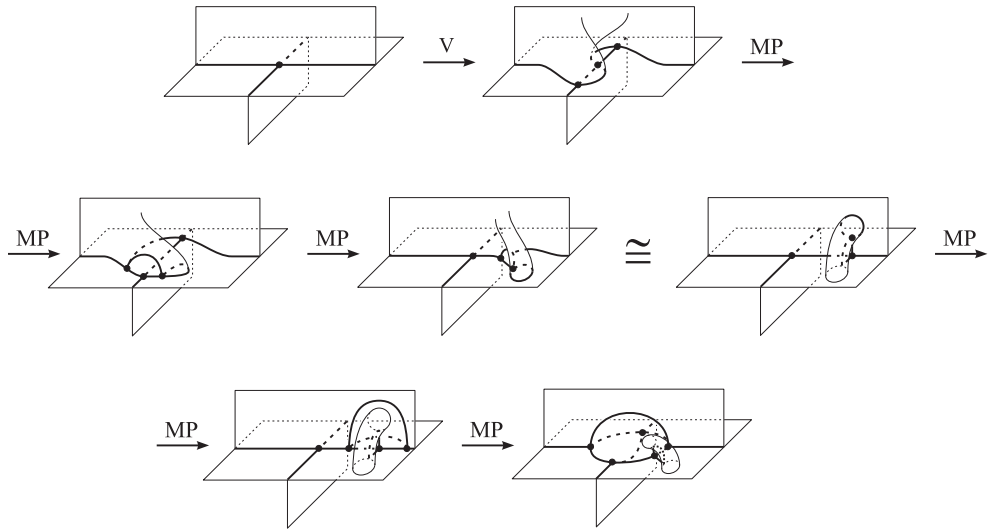}}
  \caption{Each positive C-move is a composition of V-~and MP-moves.}
  \label{C_comp_V_MP:fig}
\end{figure}
Note also that $12$ different positive C-moves can be applied at each vertex.
We are now able to prove Lemma~\ref{super_standard:lem}.

\begin{proof}[Proof of Lemma~\ref{super_standard:lem}]
First of all, for each region of $P$ that is incident to $X$ along more than one edge of $X$ (say $m$), we apply $m-1$ suitable L-moves, so that each new region is incident to at most one edge of $X$.
Call $P'$ the skeleton just obtained.

Let us call $R_i^{(j)}$, with $j=1,2,3$, the three regions incident to the component $X_i$.
Note that we have $R_i^{(j)} \neq R_l^{(k)}$ if $i \neq l$ or $j \neq k$.
Note also that each $R_i^{(j)}$ is adjacent to $R_i^{(k)}$, with $k \neq j$, along edges of $P'$ with an endpoint on $X_i$.
All the adjacencies along edges between any $R_i^{(j)}$ and $R_l^{(k)}$ that are not of this type will be called {\em bad}.

We transform $P'$, via {\rm L}-~and {\rm MP}-moves, into another standard skeleton $P'' = Q \cup Q'$, where $Q$ is a quasi-standard polyhedron without boundary and $Q' \cong X \times [0,1)$ is made up of the regions of $P$ incident to $X$.
In order to do this, it is enough to eliminate all bad adjacencies.
Suppose there is a bad adjacency between $R_i^{(j)}$ and $R_l^{(k)}$, as shown in Fig.~\ref{non_good_adj:fig};
we apply a C-move, an MP-move and an L-move, as shown in Fig.~\ref{non_good_adj_elimin:fig}.
\begin{figure}
  \centerline{\includegraphics{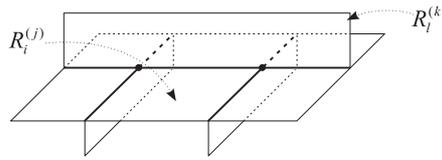}}
  \caption{A bad adjacency.}
  \label{non_good_adj:fig}
\end{figure}
\begin{figure}
  \centerline{\includegraphics{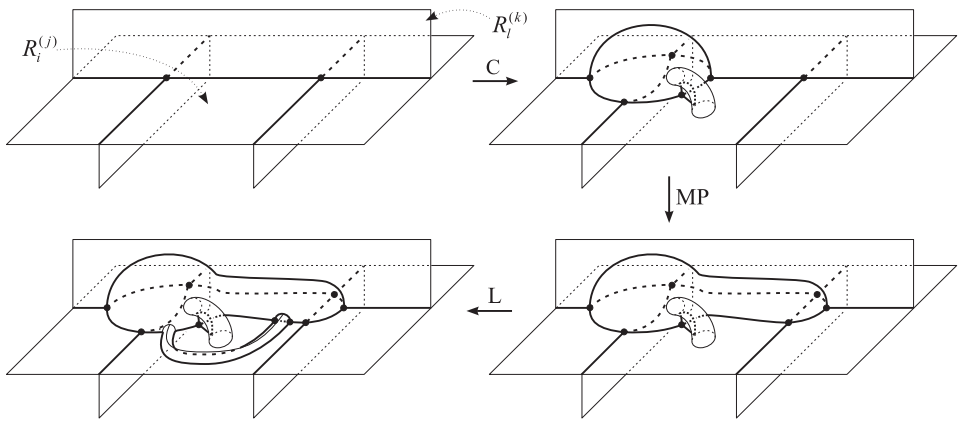}}
  \caption{Elimination of a bad adjacency.}
  \label{non_good_adj_elimin:fig}
\end{figure}
By repeating this procedure for each bad adjacency, we eliminate them (note that we do not create new ones) and hence we obtain the new skeleton $P''$ we are looking for.

The last step consists in making $Q$ standard via {\rm L}-~and {\rm MP}-moves.
First of all, we note that $Q$ is connected: in fact $P''$, which is connected, retracts by deformation onto $Q$.
We will modify $Q$ and $Q'$, but we will continue to call the polyhedra we get $Q$ and $Q'$, for the sake of shortness.
If $Q$ is a surface, we create a singularity by applying a positive C-move on a vertex of $P''$.
Note that we have modified both $Q$ and $Q'$, but we have left $Q'$ homeomorphic to $X \times [0,1)$.
If $Q$ has no vertex, we create one by applying a positive C-move as above.

We will finally transform $Q$ in order to have that all the 2-dimensional components of $Q$ are discs.
In order to divide suitably all the 2-dimensional components that are not discs, we consider a collection of disjoint simple arcs $\beta = \{\beta_1,\ldots,\beta_r\}$ that are contained in $Q$, that divide the 2-dimensional components of $Q$ into discs, and that are in general position with respect to $Q'$.
For each $\beta_i$, we apply L-moves as shown in Fig.~\ref{divide_regions:fig}.
\begin{figure}
  \centerline{\includegraphics{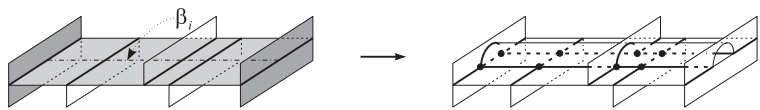}}
  \caption{How to divide the 2-dimensional components of $Q$ that are not discs (on the left, $Q$ is drawn coloured).}
  \label{divide_regions:fig}
\end{figure}
As above, note that we have modified both $Q$ and $Q'$, but we have left $Q'$ homeomorphic to $X \times [0,1)$.
Now, the polyhedron $Q$ is quasi-standard, it has vertices and all its 2-dimensional components are discs; therefore, $Q$ is standard, the skeleton just obtained is super-standard, and the proof is complete.
\end{proof}

\subsection{Calculus with fixed octopus}

The following result gives a set of moves for standard skeleta of $(M,X)$ with the same octopus.
\begin{proposition}\label{fixed_octopus:prop}
Any two standard skeleta of $(M,X)$ with the same octopus can be obtained from each other via a sequence of {\rm L}-~and {\rm MP}-moves.
\end{proposition}

Before turning into the proof, we state a corollary of Remark~\ref{one_octopus:rem} and Proposition~\ref{fixed_octopus:prop}.
\begin{corollary}
If a manifold with marked boundary $(M,X)$ has only one boundary component, then any two standard skeleta of $(M,X)$ can be obtained from each other via a sequence of {\rm L}-~and {\rm MP}-moves.
\end{corollary}

\begin{proof}[Proof of Proposition~\ref{fixed_octopus:prop}]
Let $P_1$ and $P_2$ be two standard skeleta of $(M,X)$ such that $o(P_1) = o(P_2)$.
Let $N(o)$ be a regular neighbourhood of the octopus $o = o(P_1) = o(P_2)$.
By Lemma~\ref{super_standard:lem}, we can transform each skeleton $P_i$ into a super-standard one, say $P'_i$, via L-~and MP-moves.
By virtue of Remark~\ref{invariance_octopus:rem}, we have $o(P_i) = o(P'_i) = o$.

Theorem~6.4.B of~\cite{Turaev-Viro} implies that there is a sequence of L-,~MP- and false L-moves transforming $P'_1$ into $P'_2$, where a {\em false} L-move is a negative L-move not preserving standardness (actually, with our definition of L-move, it is not an L-move).
In order to eliminate the false L-moves, we can use the same technique used in the proof of Theorem~1.2.30 of~\cite{Matveev:new:book}, by obviously generalising the setting from spines to skeleta.
Eventually, we obtain a sequence of L-~and MP-moves only, transforming $P'_1$ into $P'_2$.
The proof is complete.
\end{proof}

\section{Changing the octopus}\label{sec:changing_octopus}

In the section above, we have dealt only with moves that do not change the octopus.
The aim of this section is to describe moves that do change the octopus.

The idea is to change tentacles one by one.
We will define CR-, \Tone- and \Ttwo-moves, which are particular types of disc-replacements.
Roughly speaking, a disc-replacement on a skeleton consists in adding a disc and removing another one.
For CR-, \Tone- and \Ttwo-moves, this yields the (suitable) replacement of the tentacle dual to the added disc by the tentacle dual to the removed disc (actually, only a part of the tentacle is dual to the disc).
Here by ``suitable'' replacement we mean a change of crossing for a CR-move and a generic change of tentacle for a \Ti-move.
Hence, in order to change a crossing or a tentacle, the idea is to modify the standard skeleton to reach a configuration where the CR-, the \Tone- or the \Ttwo-move we need can be applied.
We will now go into detail.

\subsection{Disc-replacement}\label{disc_replacement:subsec}

Let $P$ be a standard skeleton of $(M,X)$.
An {\em external disc} ({\em for $P$}) is a closed disc $D$ such that $D \cap P = \partial D$, the boundary $\partial D$ is in general position
with respect to the singularities of $P$, and $D \setminus \partial D$ is embedded (in $\interior{M} \setminus P$).
The disc $D$ divides the open ball $M \setminus (P \cup \partial M)$ into two balls, say $B_1$ and $B_2$.
Let now $D'$ be a disc contained in $P \cup D$, adjacent to both $B_1$ and $B_2$, and such that the polyhedron $P' = (P \cup D) \setminus
D'$ is standard.
Note that we have $\partial P'=\partial P$.
We have that $M \setminus (P' \cup \partial M) = B_1 \cup B_2 \cup D'$ is a ball and that $P' \cap \partial M$ coincides with $X$ and $\partial P'$, so $P'$ is a standard skeleton of $(M,X)$.
The move from $P$ to $P'$ will be called a {\em disc-replacement}.

\begin{remark}\label{L_MP_disc_replacement:rem}
Each L-~and MP-move is a particular disc-replacement.
\end{remark}

\begin{remark}
A more general version of disc-replacement has been already considered by Matveev for spines~\cite{Matveev:new:book}.
It is called a {\em disc replacement move}.
\end{remark}

In the sections below, we will show that three particular disc-replacements (the CR-, the \Tone- and the \Ttwo-move) are enough to complete the sets of moves.
Therefore, we will obtain the following corollary of Theorem~\ref{only_change_tentacle:teo} below and Remark~\ref{L_MP_disc_replacement:rem}.
\begin{corollary}
Any two standard skeleta of $(M,X)$ can be obtained from each other via a sequence of disc-replacements.
\end{corollary}

\paragraph{Alteration of the octopus in a particular case}
Consider the open ball $B(P)=M\setminus(P\cup\partial M)$, which is embedded in $M$.
The closure of $B(P)$ is not embedded, so consider an unfolded version of it (say $\bar{B}(P)$).
Its boundary is divided into discs corresponding to either regions of $P$ or boundary components of $M$.
Each region of $P$ appears twice in $\partial\bar{B}(P)$, while each boundary component of $M$ appears only once.

We will now describe how the octopus changes when a particular disc-replacement is carried out.
We assume that $D$ intersects the octopus $o(P)$ once (a greater number of intersections leads to a more complicated alteration of the octopus, which will not be necessary below).
Call $\tau_i$ the tentacle intersecting $D$.
The external disc $D$ divides the ball $B(P)$ into two balls $B_1$ and $B_2$, where $B_2$ is the ball adjacent to the boundary component $C_i$ of $\partial M$ corresponding to the tentacle $\tau_i$.
Call $\alpha$ the edge dual to the region of $P$ containing $D'$, and note that it is divided by $D'$ in two sub-arcs (say $\alpha_1$ and $\alpha_2$).
Note that all the tentacles of $o(P)$ and the two arcs $\alpha_*$ are trivial in $\bar{B}(P)$, hence we can suppose that they are radial in $\bar{B}(P)$.
Moreover, we assume that $\alpha_1$ does not intersect $D$ up to isotopy.
In Fig.~\ref{octopus_ball:fig}-left a 2-dimensional picture is shown (in this example $n=3$ and $i=1$).
\begin{figure}
  \centerline{\includegraphics{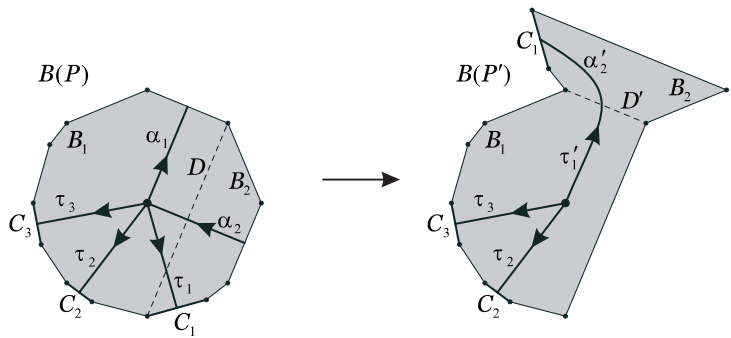}}
  \caption{How a disc-replacement alters a tentacle of the octopus (2-dimensional picture), an example with $n=3$.}
  \label{octopus_ball:fig}
\end{figure}
After the disc-replacement the ball $B(P')=M\setminus(P'\cup\partial M)$ is obtained by cutting out $B_2$ from $B(P)$ and by pasting it back to $B_1$ along $D'$.
A 2-dimensional picture of the unfolded version of $B(P')$ is shown in Fig.~\ref{octopus_ball:fig}-right.
Note that the octopus $o(P')$ has the same tentacles as $o(P)$, except for $\tau'_i$ which differs from $\tau_i$.
More precisely, $\tau'_i$ can be constructed by adding to $\alpha_1$ any arc $\alpha'_2$ that is trivial in $B_2$, starts from the endpoint of $\alpha_1$ and ends in $C_i\setminus X_i$.

\subsection{Changing the crossings}

In this section, we will describe the CR-move (which is a particular disc-re\-place\-ment) which may change the octopus, but only ``locally'': namely, this move will allow us only to change the crossings of the octopuses.

Let $o$ be an octopus and let $\varphi : [0,1] \rightarrow \interior{M}$ be a simple arc such that $\varphi([0,1]) \cap o = \{\varphi(0),\varphi(1)\}$.
Fix a trivialisation $D^3\cong D^2\times I$ of a regular neighbourhood of $\varphi([0,1])$.
We call a {\em change of crossing} the modification of $o$ in $D^3$ shown in Fig.~\ref{change_crossing:fig}.
\begin{figure}
  \centerline{\includegraphics{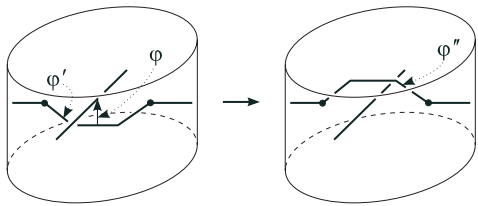}}
  \caption{The change of crossing.}
  \label{change_crossing:fig}
\end{figure}
Namely, we have replaced the arc $\varphi'$ with the arc $\varphi''$.
Note that the change of crossing modifies only a little regular neighbourhood of $\varphi([0,1])$ and does not depend on the orientation of $\varphi$.

\begin{remark}\label{one_octopus_up_crossing:rem}
If $M$ is $T\times I$ or $K\times I$, then it has only one octopus up to changes of crossings.
\end{remark}

\paragraph{{\rm CR}-move}
Let $P$ be a standard skeleton of $(M,X)$.
We call a CR{\em -move} the move shown in Fig.~\ref{CR_move:fig}.
\begin{figure}
  \centerline{\includegraphics{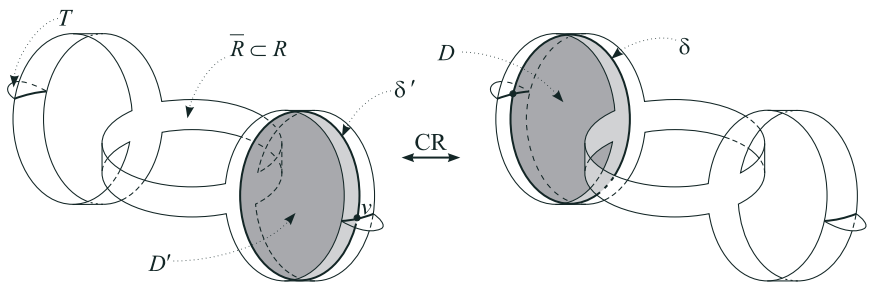}}
  \caption{The CR-move.}
  \label{CR_move:fig}
\end{figure}
Let us call $P'$ the polyhedron obtained after the CR-move.
Since this move is non-local, it must be described with some care.
Let us consider a region $D'$ incident to one vertex only, say $v$, and to one region only, say $R$, along the circle $\delta'$.
We suppose moreover that an unfolded version of $R$ appears as in Fig.~\ref{unfolded_version_R:fig}.
\begin{figure}
  \centerline{\includegraphics{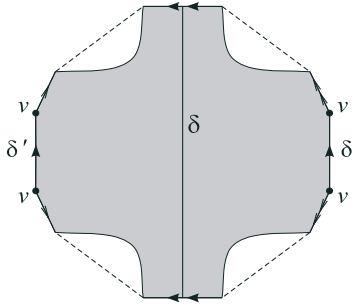}}
  \caption{Unfolded version of $R$.
  The coloured cross is an unfolded version of the part of $R$ drawn in Fig.~\ref{CR_move:fig}.}
  \label{unfolded_version_R:fig}
\end{figure}
Note that the folded version of the cross coloured in Fig.~\ref{unfolded_version_R:fig}, say $\overline{R}$, is exactly the part of $R$ involved in the move (see again Fig.~\ref{CR_move:fig}).
We suppose that the folded version is transversely orientable and that the little tongue $T$ lies on the other side of $\overline{R}$ with respect to $D'$.
The folded version of the arc $\delta$, which is a circle, bounds an external disc $D$ lying on the other side of $\overline{R}$ with respect to the little tongue $T$.
The move consists in replacing the disc $D'$ with the external disc $D$ (see Fig.~\ref{CR_move:fig}).
Note that a CR-move can lead to a non-standard polyhedron.
To avoid this loss of standardness, we will call CR-moves only those preserving standardness.
With this convention, if we apply a CR-move to a standard skeleton of $(M,X)$, the result will be another standard skeleton of $(M,X)$.
Note that each CR-move is a disc-replacement.

\paragraph{Changing the crossings}
The following result states that CR-moves are enough to change the crossings.
\begin{proposition}\label{change_crossing:prop}
Let $P_1$ and $P_2$ be standard skeleta of $(M,X)$ such that $o(P_2)$ is obtained from $o(P_1)$ via changes of crossing.
Then, $P_2$ can be obtained from $P_1$ via {\rm L}-,~{\rm MP}-~and {\rm CR}-moves.
\end{proposition}

Before turning into the proof, we state a corollary of Remark~\ref{one_octopus_up_crossing:rem} and Proposition~\ref{change_crossing:prop}.
\begin{corollary}
If $M$ is $T\times I$ or $K\times I$, then any two standard skeleta of $(M,X)$ can be obtained from each other via a sequence of {\rm L}-,~{\rm MP}-~and {\rm CR}-moves.
\end{corollary}

\begin{proof}[Proof of Proposition~\ref{change_crossing:prop}]
Obviously, it is enough to prove that if $o(P_2)$ is obtained from $o(P_1)$ via one change of crossing, then
$P_2$ can be obtained from $P_1$ via \mbox{{\rm L}-,}~{\rm MP}-~and {\rm CR}-moves.
Hence, let us suppose that $o(P_2)$ is obtained from $o(P_1)$ via one change of crossing.
We use the same notation as that of Fig.~\ref{change_crossing:fig}.

Call $\tau_1$ and $\tau_2$ the tentacles involved in the change of crossing (where $\tau_1$  contains $\varphi'$).
They are divided by $\varphi([0,1])$ in two components: call $\tau_i^{h}$ (resp.~$\tau_i^{b}$) the component of $\tau_i$ incident to the head of the octopus (resp.~to the boundary of $M$), for $i=1,2$.
We are considering the case $\tau_1\neq\tau_2$; if $\tau_1=\tau_2$ the proof is the same, except that the tentacle is divided in three components.
Consider an embedded loop $\alpha$ obtained by composition of
\begin{itemize}
\item an arc starting from the head of the octopus and running closely parallel to $\tau_1^{h}$,
\item a part of $\varphi$,
\item an arc running closely parallel to $\tau_2^{h}$ in the reverse direction and ending in the head of the octopus;
\end{itemize}
here ``closely'' means that the arc, $\tau_i^{h}$ and a small part of $\varphi$ bound a disc that does not intersect the other tentacles (for $i=1,2$), and that the internal parts of the two discs are disjoint.
By Lemma~\ref{existence:lem} and Proposition~\ref{fixed_octopus:prop}, we can suppose, up to L-~and MP-moves, that $\alpha$ is an edge of the ideal triangulation $\calT(P_1 \cup \partial M)$: let us call $R$ the region dual to $\alpha$.
Moreover, up to L-moves, we can suppose that $R$ is a disc with closure embedded in $M$ (namely, there is no self-adjacency along edges or vertices).
The local configuration now is shown in Fig.~\ref{before_change_crossing:fig}.
\begin{figure}
  \centerline{\includegraphics{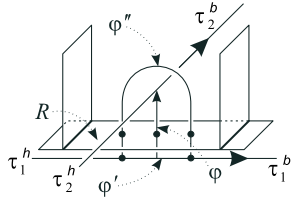}}
  \caption{The local configuration near the region $R$ dual to $\varphi$.}
  \label{before_change_crossing:fig}
\end{figure}
For the sake of simplicity, we continue to call $P_1$ the standard skeleton just obtained.

The aim is to modify $P_1$ to be able to apply a CR-move changing the crossing as desired.
We modify $P_1$ in two steps.
\begin{description}

\item{Step~1.}
Let us concentrate on the part that, in Fig.~\ref{before_change_crossing:fig}, lies over $R$.
Let us call $\overline{\varphi}''$ the part of $\varphi''$ that lies over $R$.
Up to an isotopy of $\overline{\varphi}''$, we can suppose that, by projecting $\overline{\varphi}''$ to $P_1$ along $\pi$, we obtain a disc, say $\Phi$, transverse to the singularities and to itself (see Fig.~\ref{preparation_change_crossing_1:fig}-left).
\begin{figure}
  \centerline{\includegraphics{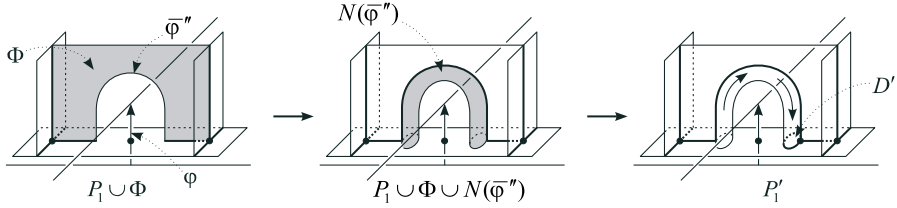}}
  \caption{How to modify $P_1$ to be able to apply a CR-move (Step~1).}
  \label{preparation_change_crossing_1:fig}
\end{figure}
Here $\pi$ is the projection of $M \setminus N(o(P_1))$ onto $P_1$.
Consider a little regular neighbourhood $N(\overline{\varphi}'')$ of $\overline{\varphi}''$, see
Fig.~\ref{preparation_change_crossing_1:fig}-centre.
The polyhedron $(P_1 \cup \Phi \cup N(\overline{\varphi}'')) \cup \partial M$ is a spine of $M$ minus a ball.
If we collapse it as shown in
Fig.~\ref{preparation_change_crossing_1:fig}-right, we obtain a quasi-standard skeleton $P'_1$ of $(M,X)$ such that $o(P'_1)=o(P_1)$.
Up to a change of $\Phi$, we can suppose that $P'_1$ is also standard.
Let us call $D'$ the little disc shown in Fig.~\ref{preparation_change_crossing_1:fig}-right.

\item{Step~2.}
Now, we concentrate on the part that, in Fig.~\ref{before_change_crossing:fig}, lies under $R$, see
Fig.~\ref{preparation_change_crossing_2:fig}.
We consider the arc analogous to $\overline{\varphi}''$ under $R$ and we project it to $P'_1$.
Let us call $\delta$ the arc just obtained.
Up to a change of the projection, we can suppose that it does not intersect the portion of $R$ near $\Phi$.
We apply L-moves along $\delta$ (see Fig.~\ref{preparation_change_crossing_2:fig}-left for an example), substituting the curve $\delta$
with another curve that intersects the singularity of $P$ in one point only.
\begin{figure}
  \centerline{\includegraphics{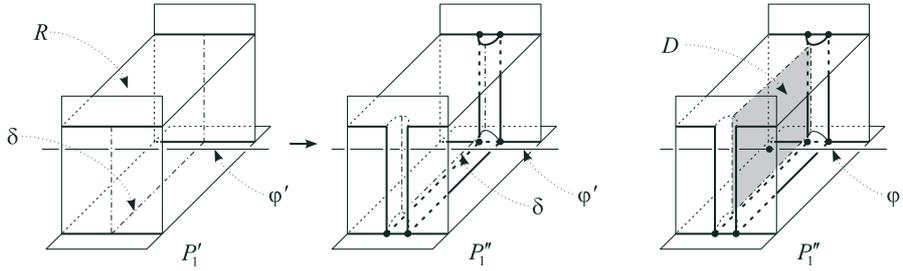}}
  \caption{Left: How to modify $P_1$ to be able to apply a CR-move (Step~2), an example.
  Right: The external disc $D$.}
  \label{preparation_change_crossing_2:fig}
\end{figure}
Let us continue to call $\delta$ the curve just obtained.
Let us call $P''_1$ the standard skeleton just obtained (obviously, $o(P''_1)=o(P'_1)=o(P_1)$) and $D$ the disc shown in Fig.~\ref{preparation_change_crossing_2:fig}-right (whose boundary contains $\delta$).
Note that $D$ is an external disc for $P''_1$.

\end{description}
By virtue of Proposition~\ref{fixed_octopus:prop}, we have that $P''_1$ can be obtained from $P_1$ via L-~and MP-moves.

Now, we are able to apply a CR-move.
The substitution of the disc $D'$ with the external disc $D$ is exactly a CR-move (the check is straightforward, so we leave it to the reader).
Call $P'_2$ the standard skeleton obtained by applying this CR-move.
In order to prove that the crossing changes as desired, note that the edge of the ideal triangulation $\calT(P''_1 \cup \partial M)$ dual to the region $D'$ is the loop $\alpha'$ obtained by composing
\begin{itemize}
\item an arc starting from the head of the octopus and running closely parallel to $\tau_1^{h}$,
\item a sub-arc of $\varphi''$,
\item an arc running closely parallel to $\varphi'$,
\item an arc running closely parallel to $\tau_1^{h}$ in the reverse direction and ending in the head of the octopus;
\end{itemize}
here ``closely'' means
\begin{itemize}
\item for the first arc, that it, $\tau_1^{h}$ and a small part of $\varphi''$ bound a disc that does not intersect the other tentacles,
\item for the third and the fourth arc, that they, $\tau_1^{h}$ and a small part of $\varphi''$ bound a disc that does not intersect the other tentacles and the disc above;
\end{itemize}
see Fig.~\ref{crossing_is_right:fig}.
\begin{figure}
  \centerline{\includegraphics{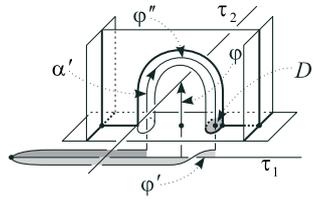}}
  \caption{The loop $\alpha'$ dual to the region $D'$.}
  \label{crossing_is_right:fig}
\end{figure}
After the CR-move, the tentacle corresponding to the component $C_1$ of $\partial M$ is obtained, up to isotopy, from $\tau_1$ by replacing the arc $\varphi'$ with $\varphi''$; see the end of Section~\ref{disc_replacement:subsec}.

In order to conclude the proof, it is enough to apply Proposition~\ref{fixed_octopus:prop} to obtain $P_2$ from $P'_2$ via L-~and MP-moves.
\end{proof}

\subsection{Generic changes of tentacle}

In this section, we will describe the \Ti-moves (two particular disc-replacements) and we will finally give the set of moves for the general case.
We need a modification of the CR-move because changes of crossing may not be enough to transform any two octopuses of $(M,X)$ into each other.

Let $o$ be an octopus, let $\varphi' : [0,1] \rightarrow o$ be a piece of a tentacle of $o$, and let $\varphi'' : [0,1] \rightarrow \interior{M}$ be a
generic simple arc such that $\varphi''([0,1]) \cap o = \{\varphi''(0),\varphi''(1)\}$ and that the endpoints of $\varphi''$ coincide with the endpoints of
$\varphi'$.
If we replace the arc $\varphi'$ of $o$ with the arc $\varphi''$, we obtain another octopus $o'$.
The transformation from $o$ to $o'$ is called a {\em change of tentacle} (see Fig.~\ref{change_tentacle:fig}).
\begin{figure}
  \centerline{\includegraphics{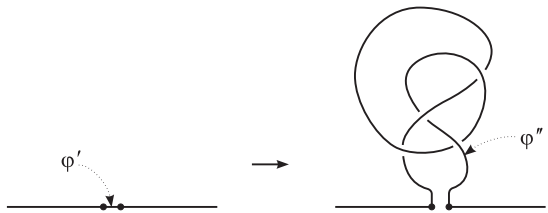}}
  \caption{A change of tentacle.}
  \label{change_tentacle:fig}
\end{figure}
Note that the change of tentacle does not depend on the orientation of $\varphi'$ and $\varphi''$.

\begin{remark}\label{change_tentacles_suffices:rem}
Changes of tentacle are enough to obtain any octopus of $(M,X)$ from a fixed one.
Namely, at most one change of tentacle for each tentacle is enough.
\end{remark}

\paragraph{\Ti-moves}
Let $P$ be a standard skeleton of $(M,X)$.
We call a \Tone{\em -move} (resp.~a \Ttwo{\em -move}) the move shown in Fig.~\ref{T_one:fig} (resp.~Fig.~\ref{T_two:fig}).
\begin{figure}
  \centerline{\includegraphics{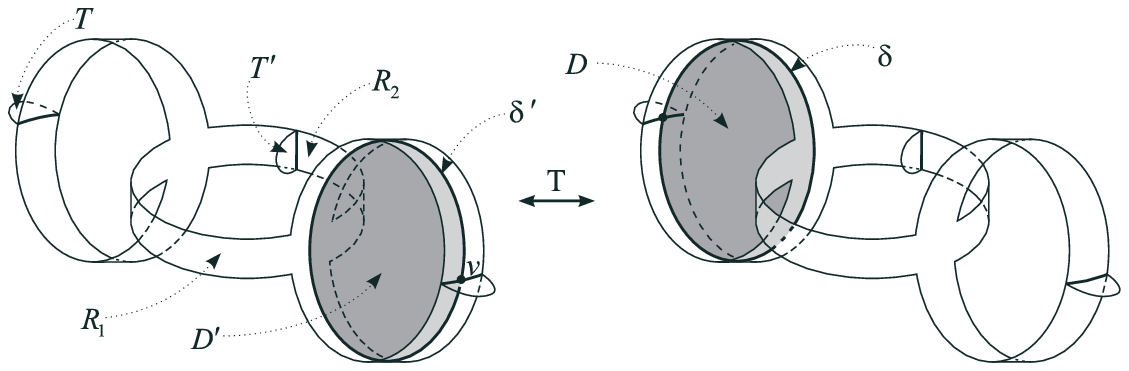}}
  \caption{The \Tone-move.}
  \label{T_one:fig}
\end{figure}
\begin{figure}
  \centerline{\includegraphics{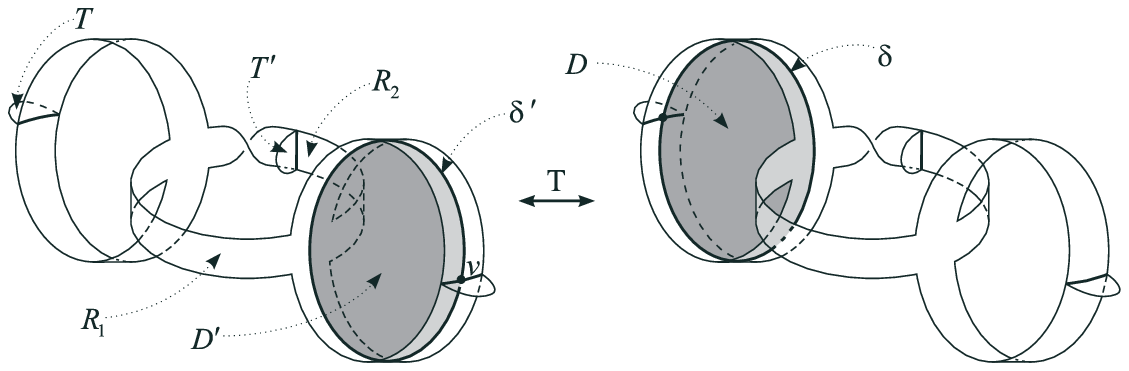}}
  \caption{The \Ttwo-move.}
  \label{T_two:fig}
\end{figure}
Let us call $P'$ the polyhedron obtained after the move.
Since this move is non-local, it must be described with some care.
Let us consider a region $D'$ incident to one vertex only, say $v$, and to two different regions, say $R_1$ and $R_2$, along the circle $\delta'$.
We suppose moreover that an unfolded version of $R_1$ and $R_2$ appears as in Fig.~\ref{unfolded_version_R_1_R_2_one:fig} (resp.~Fig.~\ref{unfolded_version_R_1_R_2_two:fig}).
\begin{figure}
  \centerline{\includegraphics{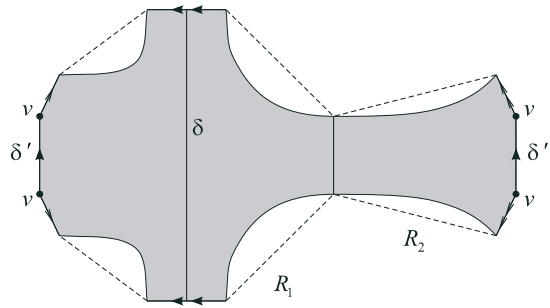}}
  \caption{Unfolded version of $R_1$ and $R_2$.
  The coloured cross is an unfolded version of the part of $R_1$ and $R_2$ drawn in Fig.~\ref{T_one:fig}.}
  \label{unfolded_version_R_1_R_2_one:fig}
\end{figure}
\begin{figure}
  \centerline{\includegraphics{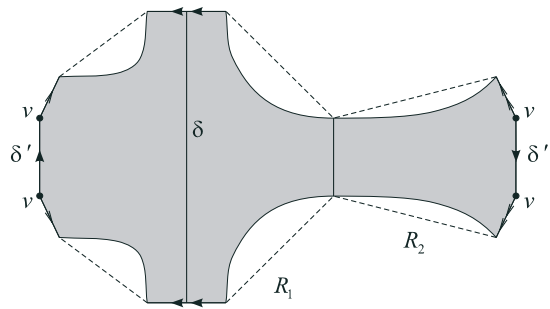}}
  \caption{Unfolded version of $R_1$ and $R_2$.
  The coloured cross is an unfolded version of the part of $R_1$ and $R_2$ drawn in Fig.~\ref{T_two:fig}.}
  \label{unfolded_version_R_1_R_2_two:fig}
\end{figure}
Note that the folded version of the cross coloured in Fig.~\ref{unfolded_version_R_1_R_2_one:fig} (resp.~Fig.~\ref{unfolded_version_R_1_R_2_two:fig}), say $\overline{R}$, is exactly the part of $R_1$ and $R_2$ involved in the move; see again Fig.~\ref{T_one:fig} (resp.~Fig.~\ref{T_two:fig}).
We suppose that the folded version is transversely orientable and that the little tongues $T$ and $T'$ lie on the other side of $\overline{R}$ with respect to $D'$.
The folded version of the arc $\delta$, which is a circle, bounds an external disc $D$ lying on the other side of $\overline{R}$ with respect to the little tongue $T$.
The move consists in replacing the disc $D'$ with the external disc $D$; see Fig.~\ref{T_one:fig} (resp.~Fig.~\ref{T_two:fig}).
Note that a \Tone-move (resp.~a \Ttwo-move) can lead to a non-standard polyhedron.
To avoid this loss of standardness, we will call \Tone-moves (resp.~\Ttwo-moves) only those preserving standardness.
With this convention, if we apply a \Tone-move (resp.~a \Ttwo-move) to a standard skeleton of $(M,X)$, the result will be another standard skeleton of $(M,X)$.
Note that each \Ti-move is a disc-replacement.

\begin{remark}
Each change of crossing is also a change of tentacle but the CR-move, which allows us to change the crossings, is not a \Ti-move, which will allow us to change the tentacles.
Actually, the difference between the two moves is deeper: the CR-move is local ({\em i.e.}~if we look at Fig.~\ref{CR_move:fig}, we note that there exists a horizontal disc which is an external disc for $P \cup D$, so the move modifies a portion of $P$ contained in a ball), while the \Ti-move may not be local.
\end{remark}

\paragraph{Strips}
We will now describe a generalisation of the procedure described in Proposition~\ref{change_crossing:prop}-Step~1 being useful in the proof of the theorem below.
Let us consider a standard skeleton $P$ of $(M,X)$.
Let $\varphi : [0,1] \rightarrow \interior{M}$ be a simple arc such that $\varphi([0,1]) \cap P = \{\varphi(0),\varphi(1)\}$.
Suppose moreover that $\varphi([0,1]) \cap o(P) = \emptyset$.
Let $\pi$ be a retraction of $M \setminus N(o(P))$ onto $P$.
Up to isotopy, we can suppose that $\pi(\varphi([0,1]))$ both intersects at least twice $S(P)$ and is in general position with respect to $S(P)$ and to itself.
Then, there exists a continuous $\Phi : [0,1] \times [0,1] \rightarrow M$ obtained by projecting $\varphi([0,1])$ along $\pi$, where $\Phi(1,0)$ and $\Phi(1,1)$ are the two intersection points mentioned above, and $\Phi(0,t)=\varphi(t)$ holds for each $t \in [0,1]$.
Such a $\Phi$ will be called a {\em strip associated to $\varphi$} (see Fig.~\ref{strip:fig}).
\begin{figure}
 \centerline{\includegraphics{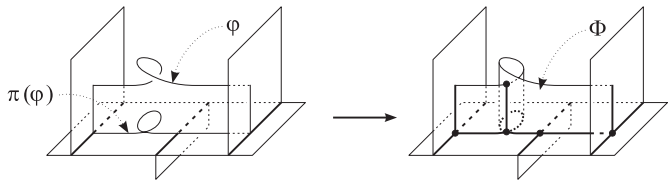}}
  \caption{A strip associated to $\varphi$, an example.}
  \label{strip:fig}
\end{figure}
Note that it may have self-intersections.
We can suppose that $\Phi$ is in general position with respect to $P$ and to itself.

Let us consider now a simple arc $\varphi : [0,1] \rightarrow \interior{M}$ such that $\varphi([0,1]) \cap P = \{\varphi(0),\varphi(s),\varphi(1)\}$ with $0<s<1$, that $\varphi([0,1]) \cap o(P) = \emptyset$, and that $\varphi$ is in general position with respect to $P$.
The arc $\varphi$ can be divided in two arcs $\varphi_1 : [0,s] \rightarrow \interior{M}$ and $\varphi_2 : [s,1] \rightarrow \interior{M}$ satisfying the hypotheses above, so there exist two strips $\Phi_1 : [0,1] \times [0,s] \rightarrow M$ and $\Phi_2 : [0,1] \times [s,1] \rightarrow M$ associated to $\varphi_1$ and $\varphi_2$, respectively.
Now, perhaps $\Phi_1$ and $\Phi_2$ do not fit together to give a continuous $\Phi : [0,1] \times [0,1] \rightarrow M$.
But, up to a move of $\Phi_2$, we can suppose that such a $\Phi$ exists and is in general position with respect to $P$ and to itself, see Fig.~\ref{fit_strip:fig}.
\begin{figure}
  \centerline{\includegraphics{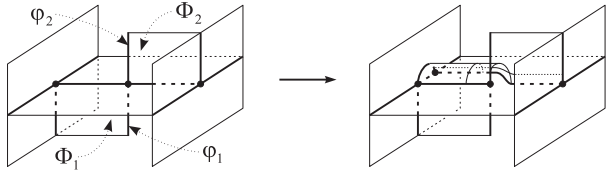}}
  \caption{How to move $\Phi_2$, so that the two strips $\Phi_1$ and $\Phi_2$ fit together.}
  \label{fit_strip:fig}
\end{figure}
As above, such a $\Phi$ will be called a {\em strip associated to $\varphi$}.
Obviously, we can generalise this technique to arcs with a generic (finite) number of intersections with $P$.

\paragraph{Changing the tentacles}
The following result states that \Ti-moves are enough to change the tentacles.
\begin{proposition}\label{change_tentacle:prop}
Let $P_1$ and $P_2$ be standard skeleta of $(M,X)$ such that $o(P_2)$ is obtained from $o(P_1)$ via changes of tentacles.
Then, $P_2$ can be obtained from $P_1$ via {\rm L}-,~{\rm MP}-~and \Ti-moves.
\end{proposition}

Before turning into the proof, we state an obvious corollary of Proposition~\ref{fixed_octopus:prop}, Proposition~\ref{change_tentacle:prop} and Remark~\ref{change_tentacles_suffices:rem}.

\begin{theorem}\label{only_change_tentacle:teo}
Any two standard skeleta of $(M,X)$ can be obtained from each other via a sequence of {\rm L}-,~{\rm MP}-~and \Ti-moves.
\end{theorem}

\begin{proof}[Proof of Proposition~\ref{change_tentacle:prop}]
Obviously, it is enough to prove that if $o(P_2)$ is obtained from $o(P_1)$ via one change of tentacle, then $P_2$ can be obtained from $P_1$ via L-,~MP-~and \Ti-moves.
Hence, let us suppose that $o(P_2)$ is obtained from $o(P_1)$ via one change of tentacle.
We use the same notation as that of Fig.~\ref{change_tentacle:fig}.

Call $\tau$ the tentacle involved in the change of tentacle.
Up to isotopy, we can suppose that $\varphi''$ intersects $P_1$ at least twice.
Let us call $t_i$, with $t_1 < \ldots < t_m$, the ``times'' at which $\varphi''$ intersects $P_1$.
Up to isotopy, we can suppose that the first and the last intersection ($\varphi''(t_1)$ and $\varphi''(t_m)$) belong to the same region (say $R$) and
that $\varphi''$ appears, near $\varphi'$, as shown in Fig.~\ref{before_change_tentacle:fig}.
Moreover, up to L-moves, we can suppose that $R$ is a disc with closure embedded in $M$ (namely, that there is no self-adjacency along edges or vertices).
\begin{figure}
  \centerline{\includegraphics{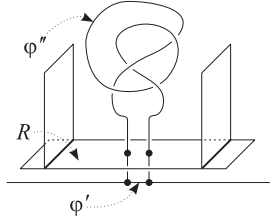}}
  \caption{The starting configuration.}
  \label{before_change_tentacle:fig}
\end{figure}

The aim is to modify $P_1$ to be able to apply a \Ti-move changing the tentacle as desired.
We modify $P_1$ in two steps.

\begin{description}

\item{Step~1.}
We repeat the same technique used in Step~2 of the proof of Proposition~\ref{change_crossing:prop} obtaining another standard skeleton $P'_1$ such that $o(P'_1) = o(P_1)$.
We obtain a curve $\delta$ intersecting the singularities of $P'_1$ only in one point, and we call $D$ the external disc shown in Fig.~\ref{preparation_change_crossing_2:fig}-right
(see Step~2 of the proof of Proposition~\ref{change_crossing:prop} for notation).

\item{Step~2.}
Let us now concentrate on the arc $\varphi''([t_1,t_m])$.
We generalise the procedure described in Step~1 of the proof of Proposition~\ref{change_crossing:prop}.
Let $\Phi$ be a strip associated to $\varphi''([t_1,t_m])$.
Up to an isotopy of $\varphi''([t_1,t_m])$ (process that requires a modification of $\Phi$), we can suppose that $\varphi''$ intersects $P$ at most thrice and that $\Phi([0,1] \times [0,1])$ does not intersect $D$.
Let us consider now a little regular neighbourhood $N(\varphi'')$ of $\varphi''([t_1,t_m])$; we have shown an example in Fig.~\ref{neighborhood_varphi:fig}.
\begin{figure}
  \centerline{\includegraphics{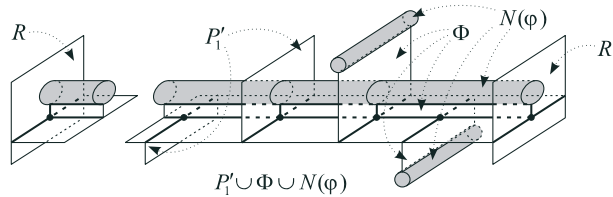}}
  \caption{A portion of the neighbourhood $N(\varphi'')$ of $\varphi''([t_1,t_m])$, an example.}
  \label{neighborhood_varphi:fig}
\end{figure}
Up to a move of the vertical regions (as shown in Fig.~\ref{move_vertical_regions:fig}), we can suppose that exactly one vertical region cuts
$N(\varphi'')$.
\begin{figure}
  \centerline{\includegraphics{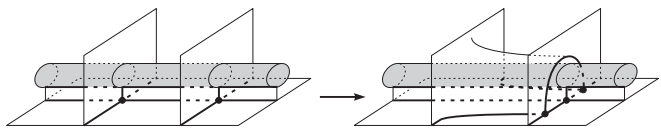}}
  \caption{How to move the vertical regions to have that exactly one vertical region cuts $N(\varphi'')$.}
  \label{move_vertical_regions:fig}
\end{figure}
The polyhedron $(P'_1 \cup \Phi \cup N(\varphi'')) \cup \partial M$ is a spine of $M$ minus a ball.
If we collapse $N(\varphi'')$ and we slightly move the vertical regions (see Fig.~\ref{standardization:fig}), we obtain a standard skeleton $P''_1$ of $(M,X)$ such that $o(P''_1) = o(P'_1)$.
\begin{figure}
  \centerline{\includegraphics{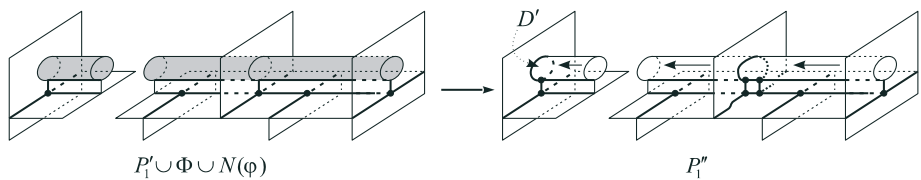}}
  \caption{How to ``standardise'' $P'_1 \cup \Phi \cup N(\varphi'')$.}
  \label{standardization:fig}
\end{figure}
Let us call $D'$ the little disc shown in Fig.~\ref{standardization:fig}-right.

\end{description}
By virtue of Proposition~\ref{fixed_octopus:prop}, we have that $P''_1$ can be obtained from $P_1$ via L-~and MP-moves.

The proof now proceeds as that of Proposition~\ref{change_crossing:prop}, so we will leave the details to the reader.
The substitution of the disc $D'$ with the external disc $D$ is exactly a \Ti-move; more precisely, it is a \Tone-move or a \Ttwo-move depending on whether the loop obtained by composing $\varphi'$ and $\varphi''$ is orientation preserving or not.
By applying this \Ti-move, we obtain a standard skeleton, say $P'_2$, and we change the tentacle as desired ({\rm i.e.}~$o(P'_2)=o(P_2)$).
In order to conclude the proof, it is enough to apply Proposition~\ref{fixed_octopus:prop} to obtain $P_2$ from $P'_2$ via L-~and MP-moves.
\end{proof}

\paragraph{Acknowledgements}
I am very grateful to Riccardo Benedetti and Carlo Petronio for the useful
discussions I have had in the beautiful period I have spent at the Galileo Galilei Doctoral School of Pisa.
I would like to thank the Department of Mathematics of the University of Salento for the nice welcome and, in particular, Prof.~Giuseppe De Cecco for his willingness.
I would also like to thank the Department of Mathematics and Applications of the University of Milano-Bicocca for the nice welcome.
Last but not least, I would like to thank the anonymous referee for very useful comments and corrections.

This research has been supported by the school of graduate studies ``Galileo Galilei'' of the University of Pisa, and afterward by the grant ``Ennio De Giorgi'' (2007-2008) from the Department of Mathematics of the University of Salento and by a Type~A Research Fellowship from the University of Milano-Bicocca.

\begin{small}

\end{small}

\noindent
\textit{Department of Mathematics and Applications\\
University of Milano-Bicocca\\
Via Cozzi, 53, I-20125, Milano, Italy}\\[5pt]
\url{gennaro.amendola@unimib.it}\\[5pt]
\url{http://www.dm.unipi.it/~amendola/}

\end{document}